\documentclass[a4paper,twoside,11pt,leqno]{article}
\usepackage{fullpage}
\usepackage{amsmath} \usepackage{amssymb} \usepackage{amsopn}
\usepackage{amsthm} \usepackage{amsfonts} \usepackage{mathrsfs}
\usepackage{mathabx}

\usepackage{verbatim}

\usepackage[OT2,T1]{fontenc}

\usepackage[pdfpagelabels,pdftex]{hyperref}
\usepackage[dvips]{graphicx} \usepackage[usenames]{color}
\usepackage[all]{xy}

\usepackage{setspace}
\usepackage{amssymb}
\usepackage{euscript}
\usepackage{cite}
\usepackage[all]{xy}
\usepackage{tabularx}

\theoremstyle{plain}
\usepackage{amsfonts}
\usepackage{graphicx}
\usepackage{amsmath}

\hypersetup{
  pdftitle={},
  pdfauthor={},
  pdfsubject={},
  pdfkeywords={},
  colorlinks=false,
  breaklinks=true,
  bookmarksopen=true,
  bookmarksnumbered=true,
  pdfpagemode=UseOutlines,
  plainpages=false}

\newcommand{\bF}{{\mathbb F}}

\newcommand{\bN}{{\mathbb N}}

\newcommand{\bQ}{{\mathbb Q}}

\newcommand{\bZ}{{\mathbb Z}}

\newcommand{\cG}{{\mathscr G}}

\newcommand{\cL}{{\mathscr L}}

\newcommand{\caO}{{\mathcal O}}

\newcommand{\caU}{{\mathcal U}}

\newcommand{\fp}{{\mathfrak p}}
\newcommand{\fq}{{\mathfrak q}}

\DeclareSymbolFont{cyrletters}{OT2}{wncyr}{m}{n}
\DeclareMathSymbol{\Sha}{\mathalpha}{cyrletters}{"58}

\DeclareMathOperator{\Isom}{Isom}

\DeclareMathOperator{\Aut}{Aut}

\DeclareMathOperator{\Spec}{Spec}

\DeclareMathOperator{\Frob}{Frob}

\DeclareMathOperator{\Gal}{G}

\newcommand{\nr}{{\rm nr}}

\newcommand{\Syl}{{\rm Syl}}

\makeatletter
\newtheorem*{rep@theorem}{\rep@title}
\newcommand{\newreptheorem}[2]{%
\newenvironment{rep#1}[1]{%
 \def\rep@title{#2 \ref{##1}}%
 \begin{rep@theorem}}%
 {\end{rep@theorem}}}
\makeatother

\newtheorem{thm}{Theorem}[section]

\newtheorem{prop}[thm]{Proposition}

\newtheorem{cor}[thm]{Corollary}

\newtheorem{lm}[thm]{Lemma}

\newreptheorem{thm}{Theorem}
\newreptheorem{prop}{Proposition}
\newreptheorem{cor}{Corollary}

\theoremstyle{definition}
\newtheorem{Def}[thm]{Definition}

\newtheorem{rem}[thm]{Remark}
\newtheorem{rems}[thm]{Remarks}

\newenvironment{pro*}[1][Proof]{{\it{#1:}} }{}

\newcommand\rar{ \rightarrow }
\newcommand\tar{ \twoheadrightarrow }
\newcommand\har{ \hookrightarrow }

\newcommand\cs{\mathop{ \rm cs}}

\newcommand\dirlim{\mathop{\underrightarrow{\lim} }}
\newcommand\prolim{\mathop{\underleftarrow{\lim} }}
\newcommand{\sm}{{\,\smallsetminus\,}}
\newcommand\N{\rm N}

\newcommand\coh{\rm H}
\newcommand\rank{\rm rk}

\newcommand\bap{\bar{\fp}}
\newcommand\baq{\bar{\fq}}
\newcommand\subsetsim{\stackrel{\subset}{\sim}}

\newcounter{absatzcounter}[section]
\numberwithin{equation}{section}

\title{On a generalization of the Neukirch-Uchida theorem}
\author{A. Ivanov\thanks{ivanov@ma.tum.de} \thanks{The author was supported by the Mathematical Center Heidelberg and by the Technische Universit\"at M\"unchen}}

\begin{document}

\maketitle

\begin{abstract}
In this paper we generalize a part of Neukirch-Uchida theorem for number fields from the birational case to the case of curves $\Spec \caO_{K,S}$ with $S$ a stable set of primes of a number field $K$. In particular, such sets can have arbitrarily small (positive) Dirichlet density. The proof consists of two parts: first one establishes a local correspondence at the boundary $S$, which works as in the original proof of Neukirch. But then, in contrast to Neukirchs proof, a direct conclusion via Chebotarev density theorem is not possible, since stable sets are in general too small, and one has to use further arguments.

\end{abstract}

\section{Introduction}

Our goal is to generalize a part of the Neukirch-Uchida theorem (\cite{Ne} Theorem 2) for number fields to schemes of the form $\Spec \caO_{K,S}$ where $K$ is a number field and $S$ a stable set of primes. Stable sets were introduced in \cite{IvStableSets}, they have positive but arbitrarily small Dirichlet density and behave in many (but not all) aspects like sets of primes of density one. In particular, many Chebotarev sets are stable. Thus the case considered in this paper is somewhere between the birational case considering $\Spec K$ and the arithmetic case considering $\Spec \caO_{K,S}$ with $S$ a finite set of primes. Clearly, the arithmetic case is much harder. The same anabelian question for function fields in one variable over a finite field, was answered by Tamagawa \cite{Ta} in the case of a finite set $S$. However, the number field analogue of his proof seems to be out of scope at the moment.

To state our main result, we recall briefly the definition of stability. Let $\lambda > 1$. Roughly speaking, a set $S$ of primes of $K$ is $\lambda$-stable, if there is a subset $S_0 \subseteq S$ and some $0 < a \leq 1$, such that the Dirichlet density of $S_0(L)$ for almost all finite subextensions of $K_S/K$ lies in the interval $[a,\lambda a)$. The condition $(\dagger)_p$ in the theorem below should be understood as 'stable and good for $p$'. It is not very restrictive.  In Section \ref{sec:stable_sets_revsec} we recall all necessary definitions and results about stable sets which we will use. 

\begin{thm}\label{thm:NU_anab_without_dec_for_stable}
For $i=1,2$, let $K_i$ be a number field and $S_i$ a set of primes of $K_i$, such that 
\begin{itemize}
\item[(a)] $K_1$ is totally imaginary and Galois over $\bQ$,
\item[(b)] for $i = 1,2$, the set $S_i$ is $2$-stable and satisfies $(\dagger)_p$ for almost all $p$,
\item[(c)] there are two odd rational primes under $S_1$,
\item[(d)] there is an odd rational prime $p$ with $S_p \subseteq S_2$ and $S_i$ satisfies $(\dagger)_p$ for $i \in \{1,2\}$.
\end{itemize}
\noindent If $\Gal_{K_1, S_1} \cong \Gal_{K_2,S_2}$, then $K_1 \cong K_2$.
\end{thm}

The important assumptions are that $K_1/\bQ$ is Galois and that $S_1,S_2$ are stable. All other assumptions are of technical nature.

The first step towards a proof of this theorem is the ``local correspondence at the boundary'', which is very similar to the birational case. Roughly speaking, it is a bijection deduced out of $\sigma$, between primes of $S_1$ and $S_2$, which respects the residue characteristic and the absolute degree of primes. It is interesting that (in contrast to the original proof of Neukirch) we neither need to know that the decomposition groups of primes in $S_i$ are the full local Galois groups, nor that there is a rational prime $p$ invertible on $\Spec \caO_{K,S}$. This last phenomenon can indeed happen for 'nice' stable sets $S$ (but not for sets with density one; cf. \cite{IvStableSets} Section 3.4).

Now the striking point is the following. In the birational case, the corresponding result of Neukirch follows directly from an established local correspondence at the boundary and an easy application of Chebotarev density theorem. The same would also work in the case of restricted ramification if $\delta_{K_i}(S_i) = 1$. But for general stable sets an analogous application of Chebotarev is impossible, since stable sets can have arbitrarily small (positive) density. Naively, this can be illustrated by the following easy consideration: the Chebotarev set $P_{M/K}(\sigma)$ do not determine the field $M$ uniquely, i.e., there are different finite Galois extensions $M,N/K$ and elements $\sigma, \tau$ in the corresponding Galois groups with $P_{M/K}(\sigma) = P_{N/K}(\tau)$. Thus beside the local correspondence one needs some further arguments to prove the theorem. Those are given in Proposition \ref{sec:uniform_bound} and Proposition \ref{prop:nonex_of_lifts}. The first involves stability property once again and the second shows that certain Galois extensions of number fields do not come by base change from smaller Galois extensions. That are these extra arguments, which require the additional technical assumptions in the theorem. 


\subsection*{Notation}

In this paper we use the same notations as in \cite{IvLocCorBound}. In particular, for a pro-finite group $G$ we denote by $G(p)$ its maximal pro-$p$ quotient and by $G_p$ a $p$-Sylow subgroup. For a subgroup $H \subseteq G$, we denote by $\N_G(H)$ its normalizer in $G$. 

For a Galois extension $M/L$ of fields, $\Gal_{M/L}$ denotes its Galois group. By $K$ we always denote an algebraic number field, that is a finite extension of $\bQ$. If $L/K$ is a Galois extension and $\bap$ is a prime of $L$, then $D_{\bap, L/K} \subseteq \Gal_{L/K}$ denotes the decomposition subgroup of $\bap$. If $\fp := \bap|_K$ is the restriction of $\bap$ to $K$, then we sometimes allow us to write $D_{\bap}$ or $D_{\fp}$ instead of $D_{\bap,L/K}$, if no ambiguity can occur. We write $\Sigma_K$ for the set of all primes of $K$ and $S,T$ will usually denote subsets of $\Sigma_K$. If $L/K$ is an extension and $S$ a set of primes of $K$, then we denote the pull-back of $S$ to $L$ by $S_L$, $S(L)$ or $S$ (if no ambiguity can occur). We write $K_S/K$ for the maximal extension of $K$, which is unramified outside $S$ and $\Gal_S := \Gal_{K,S}$ for its Galois group. Further, for $p \leq \infty$ a (archimedean or non-archimedean) prime of $\bQ$, $S_p = S_p(K)$ denotes the set of all primes of $K$ lying over $p$ and $S_f := S \sm S_{\infty}$.

\subsection*{An outline of the paper}
After recalling necessary definitions and facts about stable sets in Section \ref{sec:stable_sets_revsec}, we will in Section \ref{sec:Loccor} establish the local correspondence at the boundary for a given isomorphism of two Galois groups of the form $\Gal_{K,S}$ with $S$ stable. This is the first step towards a proof of Theorem \ref{thm:NU_anab_without_dec_for_stable}. In Sections \ref{sec:uniform_bound} and \ref{sec:nonex_of_lifts} we give two further arguments needed in its proof. Finally, in Section \ref{sec:proofofanabthms} we prove Theorem \ref{thm:NU_anab_without_dec_for_stable}.


\subsection*{Acknowledgements}

The results in this paper coincide essentially with a part of author's Ph.D. thesis \cite{IvDiss}, which was written under supervision of Jakob Stix at the University of Heidelberg. The author is very grateful to him for the very good supervision, and to Kay Wingberg, Johnannes Schmidt and a lot of other people for very helpful remarks and interesting discussions. The work on author's Ph.D. thesis was partially supported by Mathematical Center Heidelberg and the Mathematical Institute Heidelberg. Also the author is grateful to both of them for their hospitality and the excellent working conditions. 


\section{Stable sets} \label{sec:stable_sets_revsec} 

We briefly recall the concept of stability from \cite{IvStableSets}. 

\begin{Def}[part of \cite{IvStableSets} Definitions 2.4, 2.7] Let $S$ be a set of primes of $K$ and $\cL/K$ any extension. 
\begin{itemize}
\item[(i)] Let $\lambda > 1$. A finite subextension $\cL/L_0/K$ is \textbf{$\lambda$-stabilizing for $S$ for $\cL/K$}, if there exists a subset $S_0 \subseteq S$ and some $a \in (0,1]$, such that $\lambda a > \delta_L(S_0) \geq a > 0$ for all finite subextensions $\cL/L/L_0$. We say that $S$ is \textbf{$\lambda$-stable}, if it has a $\lambda$-stabilizing extension for $\cL/K$. We say that $S$ is \textbf{stable for $\cL/K$}, if it is $\lambda$-stable for $\cL/K$ for some $\lambda > 1$. We say that $S$ is ($\lambda$-)\textbf{stable}, if it is ($\lambda$-)stable for $K_S/K$.
\item[(ii)] Let $p$ be a rational prime. We say that $(S, \cL/K)$ satisfies $(\dagger)_p^{\rm rel}$, if $\mu_p \subseteq \cL$ and $S$ is $p$-stable for $\cL/K$ or $\mu_p \not\subseteq \cL$ and $S$ is stable for $\cL(\mu_p)/K$. We say that $S$ satisfies $(\dagger)_p$, if $(S, K_S/K)$ satisfies $(\dagger)_p^{\rm rel}$.
\end{itemize}  
\end{Def}

We will need the two following results about stable sets, which we take from \cite{IvStableSets}. 

\begin{thm}[\cite{IvStableSets}, Theorem 5.1(A)] \label{thm:stable_MainThm_part_A}
Let $K$ be a number field, $p$ a rational prime and $S \supseteq R$ sets of primes of $K$ with $R$ finite. Assume that $(S,K_S^R/K)$ is $(\dagger)_p^{\rm rel}$. Then
\[ K_{S,\fp}^R \supseteq \begin{cases} K_{\fp}(p), & \text{if } \fp \in S \sm R \\ K_{\fp}^{\nr}(p) & \text{if } \fp \not\in S. \end{cases} \]
\end{thm}

\begin{prop}[\cite{IvStableSets}, Proposition 5.13(ii)]  \label{prop:Sha2vanishingwoinverting}
Let $K$ be a number field, $S$ a set of primes of $K$. Let $p$ be a rational prime, $r > 0$ an integer. Assume that either $p$ is odd or $K_S$ is totally imaginary. Let $K_S/\cL/K$ be a normal subextension. Assume $(S,\cL/K)$ is $(\dagger)_p^{\rm rel}$ and $p^{\infty}|[\cL:K]$. Then
\[ \dirlim_{\cL/L/K} \Sha^2(K_S/L; \bZ/p^r\bZ) = 0. \]
\end{prop}

\begin{rems}\mbox{}
\begin{itemize}
\item[1)] Many results (e.g., such as the two quoted above, but also various Hasse-principles, Grunwald-Wang style results, finite cohomological dimension, etc.) holding for sets with density one also hold (with respect to a prime $p$) for stable sets of primes (satisfying $(\dagger)_p$). The proofs in the case of sets with density one rely heavily on the fact that various Tate-Shafarevich groups of $\Gal_{K,S}$ with finite resp. divisible coefficients vanish. This is in general not true for stable sets and the reason why many proofs still work, is that one can, using stability conditions, bound the size of Tate-Shafarevich groups, which in turn implies the vanishing of them in the limit taken over all finite subextensions of certain (infinite) subextensions $K_S/\cL/K$. 

\item[2)] Most natural examples of stable sets are almost Chebotarev sets: if $M/K$ is a finite Galois extension and $\sigma \in \Gal_{M/K}$, then the associated Chebotarev set is defined as
\[ P_{M/K}(\sigma) := \{ \fp \in \Sigma_K \colon \fp \text{ is unramified in $M/K$ and $\Frob_{\fp, M/K}$ is the conjugacy class of } \sigma \}. \]
We say that a set is almost Chebotarev, if it differs from a Chebotarev set only by a subset of density zero. They are stable in many cases (cf. \cite{IvStableSets} Corollary 3.4), and even if not, they inherit many properties of stable sets. Also a stable almost Chebotarev set satisfies $(\dagger)_p$ for almost all rational primes $p$.

\item[3)] If $S$ contains an almost Chebotarev set, then the global realization result Theorem \ref{thm:stable_MainThm_part_A} holds for $S$ with respect to all rational primes $p$ (i.e., $(K_S)_{\fp} = \overline{K_{\fp}}$), even if $S$ does not satisfy $(\dagger)_p$ for some $p$'s. The proof of this involves further arguments (cf. \cite{IvInfDens}).
\end{itemize}
\end{rems}


\section{Local correspondence at the boundary} \label{sec:Loccor} 

Generalizing results of Neukirch, we show that under certain conditions on the set $S$ of primes of $K$, the decomposition groups of primes in $S$ are intrinsically  determined by $G_{K,S}$. Since we in general do not know, whether the decomposition groups are the full local groups, we can not characterize them as absolute Galois groups of local fields and thus we have to deal with $p$-Sylow subgroups, which are of particularly simple kind. 

\subsection{Some technical preparations}

As in \cite{IvLocCorBound} we use the following notational short-cut.

\begin{Def}[\cite{IvLocCorBound} Definition 2.1] \label{def:pdecsubgroups}
A group of \emph{$p$-decomposition type} is a non-abelian pro-$p$ Demushkin group of rank $2$.
\end{Def}

Thus a group of $p$-decomposition type is of the form $\bZ_p \ltimes \bZ_p$ with $\bZ_p \har \Aut(\bZ_p) = \bZ_p^{\ast}$ injective (this follows from \cite{NSW} 3.9.9, 3.9.11). Such groups have an easy structure theory and one easily describes all their subgroups explicitly (cf. \cite{IvLocCorBound} Lemma 2.2).

\begin{lm}\label{lm:dec_intersec_with_GWT}
Let $K$ be a number field, $S$ a set of primes of $K$, $p$ a rational prime. Assume $S$ is stable and satisfies $(\dagger)_p$. If $\bap$ is a prime of $K_S$, let $D_{\bap,p} \subseteq D_{\bap}$ denote a $p$-Sylow subgroup. For any $\bap_1 \neq \bap_2 \in S(K_S)$ we have inside $\Gal_{K,S}$:
\[(D_{\bap_1,p} \colon D_{\bap_1,p} \cap D_{\bap_2,p}) = \infty \]
\noindent and 
\[
(D_{\bap_1} \colon D_{\bap_1} \cap D_{\bap_2}) = \infty.
\]
\end{lm}


\begin{proof} 
An easy index computation shows that one can go up to a finite extension of $K$ inside $K_S$. Using this, we can assume that $\bap_1|_K \neq \bap_2|_K$ and (by \cite{IvStableSets} Lemma 2.8) that for any subextensions $K_S/\cL/L/K$ with $L/K$ finite, the pair $(S,\cL/L)$ is $(\dagger)_p^{\rm rel}$. Now an application of Theorem \ref{thm:stable_MainThm_part_A} (with $R = \{ \bap_2|_K \}$) shows that $D_{\bap_2}$ lies in the kernel of the projection
\[ \Gal_{K,S} \tar \Gal_{K,S}^R, \]

\noindent whereas $D_{\bap_1,p}$ has infinite image. Thus $(D_{\bap_1,p} \colon D_{\bap_1,p} \cap D_{\bap_2,p}) \geq (D_{\bap_1,p} \colon D_{\bap_1,p} \cap V) = \infty$, where $V := \ker (\Gal_{K,S} \tar \Gal_{K,S}^R)$. The second statement follows from the first. 
\end{proof}

\begin{lm}\label{lm:not_p-adic_shit}
Let $\kappa$ be a $p$-adic field. Let $\kappa^{\prime}/\kappa$ be a Galois extension containing the maximal pro-$p$-extension of any finite subfield $\kappa^{\prime}/\lambda/\kappa$. Then $\Gal_{\kappa^{\prime}/\kappa}$ do not contain any subgroup of $p$-decomposition type.
\end{lm}

\begin{proof} Is the same as for \cite{IvLocCorBound} Lemma 3.2. \qedhere

\end{proof}

\begin{lm} \label{lm:oxujet_dostali_suki}
Let $K$ be a number field, $S$ a set of primes of $K$ and $p$ a rational prime. Assume $S$ is stable and satisfies $(\dagger)_p$. 
\begin{itemize}
\item[(i)]  Let $H_0 \subseteq D_{\bap}$ be a subgroup of $p$-decomposition type, where $\bap \in S_f$. Then $\N_{\Gal_{K,S}}(H_0) \subseteq D_{\bap}$.
\item[(ii)] Let $H$ be a subgroup of $\Gal_{K,S}$ of $p$-decomposition type. Assume there is an open subgroup $H_0 \subseteq H$ such that $H_0 \subseteq D_{\bap,p}$ for some prime $\bap \in S_f$. Then $H \subseteq D_{\bap}$.
\end{itemize}
\end{lm}

\begin{proof}
(i): Let $D_{\bap,p} \subseteq D_{\bap}$ be a $p$-Sylow subgroup, containing $H_0$. First of all, by Lemma \ref{lm:not_p-adic_shit} (which assumptions are satisfied by Theorem \ref{thm:stable_MainThm_part_A}), $\bap$ does not lie over $p$. Consequently, $D_{\bap,p}$ is also of $p$-decomposition type (it can not be pro-cyclic, since it already contains $H_0$), and hence by \cite{IvLocCorBound} Lemma 2.2, the inclusion $H_0 \subseteq D_{\bap,p}$ is open. Let $x \in \N_{\Gal_{K,S}}(H_0)$. Then $H_0 = x H_0 x^{-1} \subseteq xD_{\bap}x^{-1} = D_{x\bap}$. Thus $D_{\bap} \cap D_{x\bap} \supseteq H_0$ contains an open subgroup of $D_{\bap,p}$. Hence Lemma \ref{lm:dec_intersec_with_GWT} implies $x\bap = \bap$, i.e., $x \in D_{\bap}$.

\noindent (ii): Replacing $H_0$ by the intersection of all its $H$-conjugates, we can assume that $H_0$ is normal in $H$. Since $H_0$ is also of $p$-decomposition type, part (i) implies $H \subseteq D_{\bap}$.
\end{proof}


\subsection{Characterization of decomposition groups} 

\begin{thm} \label{thm:class_of_p_dec_subgroups}
Let $K$ be a number field and $S$ a set of primes of $K$. Assume there is a rational prime $p$ such that $S \supseteq S_{\infty}$ is stable and satisfies $(\dagger)_p$. Let $H \subseteq \Gal_{K,S}$ be of $p$-decomposition type and assume that $H$ satisfies the following technical condition: for any intermediate subgroup $H \subseteq U \subseteq \Gal_{K,S}$ with last inclusion open, $p^{\infty}|(U:\langle\langle H \rangle\rangle_U)$ where $\langle\langle H \rangle\rangle_U$ denotes the minimal closed normal subgroup of $U$, containing $H$. Then $H$ is contained in a decomposition subgroup of a unique prime in $(S_f \sm S_p)(K_S)$.
\end{thm}

\begin{proof} 
Uniqueness follows from Lemma \ref{lm:oxujet_dostali_suki} and \cite{IvLocCorBound} Lemma 2.2(ii). Let $H \subseteq \Gal_{K,S}$ be of $p$-decomposition type. By Lemma \ref{lm:oxujet_dostali_suki}(ii) it is enough to show only that an open subgroup of $H$ is contained in a decomposition group of a prime in $S_f \sm S_p$. Hence by \cite{IvStableSets} Lemma 2.8 we can assume that for all subextensions $K_S/\cL/K$, the pair $(S, \cL/K)$ is $(\dagger)_p^{\rm rel}$. Further, if $p=2$, by Theorem \ref{thm:stable_MainThm_part_A} $K_S$ is totally imaginary, hence we can assume that $K$ is totally imaginary in this case. For any $H \subseteq U \subseteq \Gal_{K,S}$ with last inclusion open consider the restriction map $\coh^2(U, \bZ/p\bZ) \rar \bigoplus_{\fp \in S(U)} \coh^2(D_{\fp,K_S/K_S^U}, \bZ/p\bZ)$. The main observation is that 
\[ \dirlim_{H \subseteq U \subseteq \Gal_{K,S}} \Sha^2(U, \bZ/p\bZ) = 0, \]
\noindent as by our assumption and by Proposition \ref{prop:Sha2vanishingwoinverting}, for any class $\alpha \in \Sha^2(U,\bZ/p\bZ)$ there is a subgroup $U \supseteq V \supseteq \langle\langle H \rangle\rangle_U \supseteq H$ with first inclusion open, such that the image of $\alpha$ in $\Sha^2(V,\bZ/p\bZ)$ is zero. Thus passing to the direct limit over all open $U$ containing $H$ we obtain an injection:

\[ \bZ/p\bZ \cong \coh^2(H, \bZ/p\bZ) \har \prod_{\fp \in S(M)} \coh^2(D_{\fp,K_S/M}, \bZ/p\bZ).\]

\noindent where $M := K_S^H$. Hence there is a prime $\fp \in S(M)$ with $\coh^2(D_{\fp,K_S/M}, \bZ/p\bZ) \neq 0$, this prime is non-archimedean, since $K$ is totally imaginary if $p=2$ and the proof can be finished as in the original paper of Neukirch \cite{Ne} Theorem 1 ( also cf. \cite{NSW} 12.1.9).
\end{proof}

\begin{cor}\label{cor:GS_det_decsgrs_intri} Let $K$ be a number field and $S$ a set of primes of $K$. Assume there is a rational prime $p$ such that the following hold:
\begin{itemize}
\item[(i)] $S \supseteq S_{\infty}$ is stable and satisfies $(\dagger)_p$
\item[(ii)] for each $\fp \in S_f$, we have $\mu_p \subseteq K_{S,\fp}$.
\end{itemize}
Then the group $G_{K,S}$ given together with the prime $p$ determines intrinsically the decomposition subgroups of primes in $S_f \sm S_p$.
\end{cor}

\begin{proof} The proof works exactly as in \cite{IvLocCorBound} Section 3.4. For convenience we repeat it here (except for some technical details). First of all, since $\mu_p \subset K_{S,\fp}$ for any $\fp \in S_f$, it follows from Theorem \ref{thm:stable_MainThm_part_A} that for any $\bap \in (S_f \sm S_p)(K_S)$, the composition

\[ \cG_{\fp,p} \har \cG_{\fp} \tar D_{\bap} \har \Gal_{K,S} \]

\noindent is injective, or with other words, the $p$-Sylow subgroup of $D_{\bap}$ is of $p$-decomposition type. For any $U \subseteq \Gal_{K,S}$ open (and small enough) we claim the equality of the following subsets of the set of all subgroups of $p$-decomposition type inside $U$:
\[ \Syl_p(U,S_f \sm S_p) =  \left\{ H \subseteq U \colon \begin{aligned} &H \text{ is a subgroup of $p$-decomposition type satisfying}  \\ &
\text{the technical condition in Theorem \ref{thm:class_of_p_dec_subgroups}} \\ & \text{and maximal of this type}  \end{aligned} \right\}, \]

\noindent where $\Syl_p(U,S_f \sm S_p)$ denotes the set of all $p$-Sylow subgroups of decomposition subgroups of primes in $S_f \sm S_p$ of the field $K_S^U$. Indeed, Theorem \ref{thm:class_of_p_dec_subgroups} assures that any group in the right set is contained in a decomposition group of a prime in $S_f \sm S_p$, and by maximality it has to be a $p$-Sylow subgroup. Conversely,  any group $H$ lying in the left set is of $p$-decomposition type, satisfies the technical property from Theorem \ref{thm:class_of_p_dec_subgroups} (by Lemma \ref{lm:dec_intersec_with_GWT}) and is maximal with these properties. Indeed, to prove maximality assume $H \subseteq H^{\prime}$ with $H^{\prime}$ of $p$-decomposition type. This inclusion has to be open by \cite{IvLocCorBound} Lemma 2.2(ii) and thus by Lemma \ref{lm:oxujet_dostali_suki}(ii), $H^{\prime}$ is contained in the same decomposition group as $H$. Since both are pro-$p$-groups and $H$ is a $p$-Sylow subgroup, we get $H = H^{\prime}$, proving the maximality of $H$. 

Thus the data $(\Gal_{K,S},p)$ determine intrinsically the set $\Syl_p(U,S_f \sm S_p)$ for any $U \subseteq \Gal_{K,S}$ open. $U$ acts on this set by conjugation. We have an $U$-equivariant surjection 

\[ \psi \colon \Syl_p(U, S_f \sm S_p) \tar (S_f \sm S_p)(U) \]

\noindent ($U$ acts trivially on the right side), which sends $H$ to the prime $\bap|_L$, such that $H \subseteq D_{\bap,K_S/L}$, where $L := K_S^U$. By \cite{IvLocCorBound} Lemma 3.9 we have a purely group theoretical criterion for two elements on the left side to lie in the same fiber of this surjection, which allows us to reconstruct the set $(S_f \sm S_p)(U)$. For any inclusion $V \har U$ of open subgroups of $\Gal_{K,S}$, we have (a priori non-canonical) maps $\Syl_p(V, S_f \sm S_p) \rar \Syl_p(U, S_f \sm S_p)$, and via $\psi$ they induce the restriction-of-primes maps $(S_f \sm S_p)(V) \rar (S_f \sm S_p)(U)$. Finally, if $U \subseteq \Gal_{K,S}$ is normal, the $\Gal_{K,S}$-action by conjugation on $\Syl_p(U, S_f \sm S_p)$ induces via $\psi$ the natural $\Gal_{K,S}$-action on $(S_f \sm S_p)(U)$ by permuting the primes. In this way we have reconstructed the projective system of $\Gal_{K,S}$-sets $\{ (S_f \sm S_p)(U) \colon U \subseteq U_0, U \triangleleft \Gal_{K,S} \}$, where $U_0 \subseteq \Gal_{K,S}$ is some open subgroup. Now the decomposition subgroups of primes in $S_f \sm S_p$ are exactly the stabilizers in $\Gal_{K,S}$ of elements in the $\Gal_{K,S}$-set $\prolim_{U \subseteq U_0, U \triangleleft \Gal_{K,S}} (S_f \sm S_p)(U)$. \qedhere
\end{proof}

\begin{rem}\mbox{}
\begin{itemize}
\item[(i)] It is remarkable that Theorem \ref{thm:class_of_p_dec_subgroups} and Corollary \ref{cor:GS_det_decsgrs_intri} hold even if $\bN(S) = \{1\}$, i.e, if there is no rational prime invertible on $\Spec \caO_{K,S}$. Indeed, there are examples of stable sets with arbitrarily small density satisfying $(\dagger)_p$ for all $p$ and $\bN(S) = \{1\}$ (\cite{IvStableSets} Section 3.4)
\item[(ii)] Observe that $(\dagger)_2$ is equivalent to being $2$-stable. In particular, the assumptions of Corollary \ref{cor:GS_det_decsgrs_intri} in the case $p = 2$ reduce to '$S$ is $2$-stable and $S \supseteq S_{\infty}$'.
\end{itemize}
\end{rem}


\subsection{Local correspondence at the boundary}

\begin{Def} For $i=1,2$, let $K_i$ be a number field, $S_i$ a set of primes of $K_i$ and let 

\[ \sigma \colon \Gal_{K_1, S_1} \stackrel{\sim}{\rar} \Gal_{K_2,S_2} \] 

\noindent be a (topological) isomorphism. If $U_1$ is a closed subgroup of $\Gal_{K_1, S_1}$ with fixed field $L_1$, we write $U_2$ for $\sigma(U_1)$ and $L_2$ for its fixed field, etc. We say that the \textbf{local correspondence at the boundary} holds for $\sigma$, if the following conditions are satisfied: 

\begin{itemize}
 \item[(i)] For $i = 1,2$, there is a finite exceptional set $S_i^{\rm ex} \subseteq S_i$, such that for any $\bap_1 \in (S_{1,f} \sm S_1^{\rm ex})(K_{1,S_1})$, there is a unique prime $\sigma_{\ast}(\bap_1) \in (S_{2,f} \sm S_2^{\rm ex})(K_{2,S_2})$, with $\sigma(D_{\bap_1}) = D_{\sigma_{\ast}(\bap_1)}$, such that $\sigma$ induces a bijection

\[ \sigma_{\ast} \colon (S_{1,f} \sm S_1^{\rm ex})(K_{1,S_1}) \stackrel{\sim}{\longrightarrow} (S_{2,f} \sm S_2^{\rm ex})(K_{2,S_2}) \]

\noindent which is Galois-equivariant, i.e.,
\[\sigma_{\ast}(g \bap_1) = \sigma(g)\sigma_{\ast}(\bap_1) \]

\noindent for each $g \in \Gal_{K_1,S_1}$ and $\bap_1 \in S_{1,f}(K_{1,S_1})$. In particular, for any finite subextension $L_1$ of $K_{1,S_1}/K_1$ with corresponding open subgroup $U_1 \subseteq \Gal_{K_1,S_1}$, if two primes $\bap_1$, $\baq_1 \in S_{1,f}(K_{1,S_1})$ restrict to the same prime of $L_1$, then also $\sigma_{\ast}(\bap_1), \sigma_{\ast}(\baq_1)$ restrict to the same prime of $L_2$, and hence $\sigma_{\ast}$ induces a bijection 

\[ \sigma_{\ast,U_1} \colon (S_{1,f} \sm S_1^{\rm ex})(L_1) \stackrel{\sim}{\longrightarrow} (S_{2,f} \sm S_2^{\rm ex})(L_2). \]

\noindent 

\item[(ii)] For all $K_{1,S_1}/L_1/K_1$ finite with corresponding subgroup $U_1 \subseteq \Gal_{K_1,S_1}$ and for all but finitely many primes $\fp_1 \in (S_{1,f} \sm S_1^{\rm ex})(L_1)$, the residue characteristics and the local degrees of $\fp_1$ and $\sigma_{\ast, U_1}(\fp_1)$ are equal. 
\end{itemize}
\end{Def}

\begin{cor}\label{cor:loccorforstableingeneral}
For $i =1,2$, let $K_i$ be a number field and $S_i$ a stable set of primes. Assume that $K_{i,S_i}$ is totally imaginary and that $S_i$ satisfies $(\dagger)_p$ for almost all rational primes $p$ and in particular for $p=2$. Let 
\[ \sigma \colon \Gal_{K_1,S_1} \stackrel{\sim}{\longrightarrow} \Gal_{K_2,S_2} \] 
\noindent be an isomorphism. Then the local correspondence at the boundary holds for $\sigma$ and moreover, one can choose $S_i^{\rm ex}$ to be the set of 2-adic primes in $S_i$. More precisely, for any $U_1 \subseteq \Gal_{K_1,S_1}$, $\sigma_{\ast,U_1}$ preserves the residue characteristic and the absolute degree of all primes $\fp \in S_1(U_1)$, whose residue characteristic $\ell$ is odd and such that $S_i$ satisfies $(\dagger)_{\ell}$ for $i=1,2$. 
\end{cor}

\begin{proof} To avoid notational problems, let us exceptionally denote by $S_{\text{2-adic}}(L)$ the set of 2-adic primes of a number field $L$. We apply Corollary \ref{cor:GS_det_decsgrs_intri} to $(K_i,S_i, p=2)$ for $i=1,2$. It shows that $\sigma$ maps decomposition groups of primes in $S_{1,f} \sm S_{2-adic}$ to decomposition groups of primes in $S_{2,f} \sm S_{2-adic}$. Thus we can define $\sigma_{\ast}(\bap_1)$ by the equality
\[ D_{\sigma_{\ast}(\bap_1)} = \sigma(D_{\bap_1}). \]

\noindent The Galois-equivariance of $\sigma_{\ast}$ is straightforward. It remains to show that for any finite subextension $K_{1,S_1}/L_1/K_1$ with corresponding open subgroup $U_1$ and for any $\fp_1 \in S_{1,f}(L_1)$ with residue characteristic $\ell$, which is odd and such that $S_i$ satisfies $(\dagger)_{\ell}$ for $i = 1,2$, the map $\sigma_{\ast, U_1}$ preserves the residue characteristic and the local absolute degree. For any such $\ell$ and any $\bap \in S_{1,f}(K_{1,S_1})$ with residue characteristic $\ell$, Theorem \ref{thm:stable_MainThm_part_A} implies that the maximal $\ell$-extension of $K_{i,\fp}$ is realized by $K_{i,S_i}$. Let $\fp \in S_i(L_i)$ be a prime with residue characteristic $\ell$ and $\bap$ an extension to $K_{i,S_i}$. Lemma \ref{lm:someanablocalinform} shows that $D_{\bap,K_S/L_i}$ encodes the information about the residue characteristic and the absolute degree of $\fp$. Thus $\sigma_{\ast, L_1}$ preserves the residue characteristic and the absolute degree of all primes in $S_{1,f}$ with residue characteristic $\ell$ being odd and such that $S_i$ satisfies $(\dagger)_{\ell}$ for $i=1,2$.
\end{proof}

\begin{lm} \label{lm:someanablocalinform}
Let $\kappa$ be a local field with characteristic zero and some residue characteristic $\ell$ and $\lambda/\kappa$ a Galois extension with Galois group $D$, which contains the maximal pro-$\ell$-extension of $\kappa$. Then $D$ encodes the information about $\ell$ and $[\kappa:\bQ_{\ell}]$.
\end{lm}

\begin{proof}
Let $\cG_{\kappa}$ be the absolute Galois group of $\kappa$. We have a surjection $\pi \colon \cG_{\kappa} \tar D$, and for any open $U \subseteq \cG_{\kappa}$ with $\cG_{\kappa}/U$ an $\ell$-group, surjections $U \tar \pi(U) \tar U^{(\ell)}$, which for any rational prime $p$ induce injections 
\[ \coh^1(U^{(\ell)}, \bZ/p\bZ) \har \coh^1(\pi(U), \bZ/p\bZ) \har \coh^1(U, \bZ/p\bZ).\] 

\noindent For all primes $p \neq \ell$, the $\bF_p$-dimension of the space on the right (and hence also in the middle) is bounded by $2$, and for $p = \ell$, the dimension of the space on the left gets arbitrary big, if $U$ gets arbitrary small among all subgroups $U \subseteq \cG_{\kappa}$ such that $\cG_{\kappa}/U$ is an $\ell$-group. Thus the residue characteristic $\ell$ is equal to the unique prime $p$, such that $\dim_{\bF_p} \coh^1(V,\bZ/p\bZ)$ is unbounded as $V$ varies over all open subgroups of $D$, such that $D/V$ is a $p$-group. Further, 
\[ [\kappa:\bQ_{\ell}] = \chi_{\ell}(\cG_{\kappa}(\ell),\bZ/\ell\bZ) = \chi_{\ell}(D(\ell),\bZ/\ell\bZ). \qedhere \]
\end{proof}

\begin{rem}
The proofs (of Theorem \ref{thm:class_of_p_dec_subgroups}, Corollary \ref{cor:GS_det_decsgrs_intri} and Corollary \ref{cor:loccorforstableingeneral}) would be less technical if one would assume the stronger condition: '$K_S$ realize the maximal local extension at each $\fp \in S_f$' on the involved stable sets $S$. It is satisfied in the following cases.
\begin{itemize}
\item[(i)]  If $S$ is defined over a totally real subfield, (not necessarily stable), and $S \supseteq S_{\infty} \cup S_{p_1} \cup S_{p_2}$ for two different rational primes $p_1,p_2$ (by \cite{CC} Remark 5.3(i)). 
\item[(ii)]  If $S$ is stable and satisfies $(\dagger)_p$ for all rational $p$'s (by \cite{IvStableSets}).
\item[(iii)] If $S$ contains an almost Chebotarev set (by \cite{IvInfDens}). 
\end{itemize}
We conjecture that it is true in general if $S$ is stable.
\end{rem}

\section{Anabelian geometry of curves $\Spec \caO_{K,S}$ with $S$ stable}\label{sec:anabgeomofstablesets}


\subsection{Uniform bound}\label{sec:uniform_bound} 

Besides the local correspondence on the boundary, the following argument plays a central role in the proof of Theorem \ref{thm:NU_anab_without_dec_for_stable}. From now on, we consider all occurring fields to be subfields of a fixed algebraic closure $\overline{\bQ}$ of $\bQ$. 

\begin{prop}[Uniform bound]\label{lm:uniform_bound} For $i=1,2$, let $K_i$ be a number field, $S_i$ a set of primes of $K_i$ and let

\[ \sigma \colon \Gal_{K_1, S_1} \stackrel{\sim}{\rar} \Gal_{K_2,S_2} \] 

\noindent be an isomorphism. Assume that the local correspondence at the boundary holds. Assume that $S_1$ is stable. Then there is some $N > 0$, such that for all (not necessarily finite) intermediate subfields $K_{1,S_1}/M_1/K_1$, such that $M_1$ is normal over $\bQ$, one has $[M_1 : M_1 \cap M_2] < N$, where $M_2/K_2$ corresponds to $M_1/K_1$ via $\sigma$.
\end{prop}

\begin{lm} \label{lm:cofinalvslim}
Let $\kappa$ be a field. If $(V_i)_{i \in I}$ is a cofiltered system of $\kappa$-vector spaces, such that $\dim_{\kappa} V_i < n$, and $V := \dirlim_I V_i$ , then $\dim_{\kappa} V < n$. 
\end{lm}
\begin{proof}[Proof of Lemma \ref{lm:cofinalvslim}]
For any $n$ vectors in $V$ there is an $i \in I$, such that these vectors has preimages in $V_i$. These preimages are linearly dependent. Hence their images in $V$ are linearly dependent.
\end{proof}

\begin{proof}[Proof of Proposition \ref{lm:uniform_bound}]
Since $S_1$ is stable, by \cite{IvStableSets} Proposition 2.6, there is some $N > 0$, such that $\delta_{L_1}(S_1) > N^{-1}$ for all finite subfields $K_{1,S_1}/L_1/K_1$. Let $M_1$ be a subextension of $K_{1, S_1}/K_1$, such that $M_1/\bQ$ is normal. By Lemma \ref{lm:cofinalvslim} and since $M_1$ is a union of finite extensions of $K_1$, which are normal over $\bQ$, we can assume that $M_1/K_1$ finite. Let 
\[S_1^{\prime} := S_1(M_1) \cap \cs(M_1/\bQ)(M_1). \] 

\noindent Since $M_1/\bQ$ is normal, $\delta_{M_1}(\cs(M_1/\bQ)(M_1)) = 1$ and hence

\[\delta_{M_1}(S_1^{\prime}) = \delta_{M_1}(S_1) > N^{-1}. \] 

\begin{lm}\label{lm:compdensity_for_uniform_bound} Let $S_2^{\prime} := \sigma_{\ast}(S_1^{\prime})$.  Then 
\begin{itemize}
 \item[(i)] $\delta_{M_2}(S_2^{\prime}) = \delta_{M_1}(S_1^{\prime})$.
\item[(ii)] $S_2^{\prime} \subsetsim \cs(M_1 M_2 / M_2)$.
\end{itemize}
\end{lm}

\begin{proof}[Proof of Lemma \ref{lm:compdensity_for_uniform_bound}]

(i) follows from the local correspondence at the boundary by explicitly computing the density, since $\sigma_{\ast}$ preserves the residue characteristic and the absolute degree of almost all primes in $S_1^{\prime}$. 

(ii): Let $\fp_1 \in S_1^{\prime}$ be such that $\sigma_{\ast}$ preserves the residue characteristic and the absolute degree of $\fp_1$. Let $\fp_2 := \sigma_{\ast}(\fp_1) \in S_2^{\prime}$ and $\fp := \fp_2|_{M_1 \cap M_2}$. The fiber $\caO_{M_1 M_2} \otimes_{\caO_{M_2}} \kappa(\fp_2)$ over $\fp_2$ in $\Spec \caO_{M_1 M_2}$ is isomorphic to $(\caO_{M_1} \otimes_{\caO_{M_1 \cap M_2}} \kappa(\fp)) \otimes_{\kappa(\fp)} \kappa(\fp_2)$. By assumption, we have $\fp_2|_{\bQ} = \fp_1|_{\bQ} \in \cs(K_1/\bQ)$ and hence $\fp \in \cs(M_1/\bQ)(M_1 \cap M_2) \subseteq \cs(M_1/M_1 \cap M_2)$. This implies that $\caO_{M_1} \otimes_{\caO_{M_1 \cap M_2}} \kappa(\fp)$ is isomorphic to a product of copies of $\kappa(\fp)$. Thus we obtain 
\[\caO_{M_1M_2} \otimes_{\caO_{M_1}} \kappa(\fp_2) \cong \prod \kappa(\fp_2),\] 

\noindent i.e., $\fp_2$ is completely decomposed in $M_1M_2$. 
\end{proof}

Using Lemma \ref{lm:compdensity_for_uniform_bound} and the normality of $M_1 M_2 / M_2$, we obtain:
\begin{eqnarray*}
[M_1 : M_1 \cap M_2]^{-1} &=& [M_1 M_2 : M_2]^{-1} \\ 
&=& \delta_{M_2}( \cs(M_1 M_2 / M_2 )) \\
&\geq& \delta_{M_2}(S_2^{\prime}) \\
&=& \delta_{M_1}(S_1^{\prime}) \\
&>& N^{-1}.
\end{eqnarray*}
\noindent This proves Proposition \ref{lm:uniform_bound}.
\end{proof}


\subsection{Non-existence of lifts}\label{sec:nonex_of_lifts} 

Last but not least, Proposition \ref{prop:nonex_of_lifts} proven in this section provides the last argument which we need in the proof of Theorem \ref{thm:NU_anab_without_dec_for_stable}. Let $L/K$ be a Galois extension of global fields. We want to study, under which conditions there is no Galois extension $L_0/K_0$, such that $L/K$ is a base change of $L_0/K_0$, i.e., $K_0 = K \cap L_0$ and $L = K L_0$.

\begin{prop}\label{prop:nonex_of_lifts}
Let $K, L_0$ be two linearly disjoint Galois extensions of a global field $K_0$, and set $L = K L_0$. Assume one of the following holds:

\begin{itemize}
\item[(a)] \begin{itemize}
	   \item $K$ is a totally imaginary number field and 
	   \item $L = K_{S_p}(p)$ for some prime number $p$, or
           \end{itemize}
\item[(b)] There is a prime $\fp$ of $K_0$, which is completely split in $K$, such that for any $\bap_1,\bap_2 \in S_{\fp}(L)$ with $\bap_1|_K \neq \bap_2|_K$, we have $D_{\bap_1, L/K} \neq D_{\bap_2, L/K}$. 

\end{itemize}

\noindent Then $K = K_0$. 
\end{prop}

We will only use part (a) of this proposition.

\begin{proof}
Assume (a) holds. Then $L/K$ and $L_0/K_0$ are both Galois with Galois group isomorphic to $\Gal_{K,S_p}(p)$. By \cite{NSW} 10.3.20, the number of independent $\bZ_p$-extensions of $K$ satisfies  
\[ \rank_{\bZ_p} \Gal_{K,S_p}^{\rm ab}(p) \geq r_2(K) + 1. \] 

\noindent Since $\Gal_{L/K} \cong \Gal_{L_0/K_0}$, the field $K_0$ has at least $r_2(K) + 1$ independent $\bZ_p$-extensions. Assume $K \neq K_0$. Then $[K : K_0] \geq 2$, and since $K$ is totally imaginary, we obtain:
\[ r_2(K) + 1 =  \frac{[K : \bQ]}{2} + 1 \geq [K_0 : \bQ ] + 1 > [K_0 : \bQ]. \]
But by \cite{NSW} 10.3.20, the number of independent $\bZ_p$-extensions of $K_0$ is $\leq [K_0 : \bQ]$. This is a contradiction, hence $K = K_0$ (notice that we nowhere made use of Leopoldt's conjecture!).

Assume (b) holds. Let $\psi \colon \Gal_{L/K} \stackrel{\sim}{\rar} \Gal_{L_0/K_0}$ denote the canonical isomorphism. Assume there are two different primes $\fp_1 \neq \fp_2$ in $K$ over $\fp$. Let $\fq$ be some prime of $L_0$ over $\fp$. One can chose primes $\bap_i \in S_{\fp}(L)$, such that $\bap_i|_K = \fp_i$ and $\bap_i|_{L_0} = \fq$. As $\fp_1,\fp_2$ are split over $K_0$, we obtain that $\psi$ maps $D_{\bap_i, L/K}$ isomorphically to $D_{\fq, L_0/K_0}$. But by assumption $D_{\bap_1, L/K} \neq D_{\bap_2, L/K}$, hence $D_{\fq, L_0/K_0} = \psi(D_{\bap_1, L/K}) \neq \psi(D_{\bap_2,L/K}) = D_{\fq,L_0/K_0}$, which is a contradiction. Thus there is only one prime over $\fp$ in $K$, and since $\fp$ is completely split, we obtain $[K:K_0] = 1$. 
\end{proof}


\subsection{Proof of Theorem \ref{thm:NU_anab_without_dec_for_stable}} \label{sec:proofofanabthms}

By assumption (b) in Theorem \ref{thm:NU_anab_without_dec_for_stable} and Corollary \ref{cor:loccorforstableingeneral} the local correspondence at the boundary holds for $\sigma$: for any open subgroup $U_1 \subseteq \Gal_{K_1,S_1}$ with fixed field $L_1$, $\sigma$ induces a functorial bijection

\[ \sigma^{\ast}_{U_1} \colon (S_{1,f} \sm S_{2-adic})(L_1) \stackrel{\sim}{\rar} (S_{2,f} \sm S_{2-adic})(L_2), \]

\noindent which preserves the residue characteristic and the absolute degree of all primes in \\ $(S_{1,f} \sm (S_{2-adic} \cup T))(L_1)$, with $T := \{\ell \colon S_1 \text{ or } S_2 \text{ is not } (\dagger)_{\ell} \}$. We obtain $[K_2 : \bQ] \leq [K_1 : \bQ]$ from this. Indeed, by assumption (d) there is an odd prime $p$ with $S_p \subseteq S_2$ and $p \not\in T$, and hence $\sigma_{\ast}$ preserves the residue characteristic and the absolute degree of primes in $S_p(K_1) \cap S_1$ by Corollary \ref{cor:loccorforstableingeneral}. Hence $\sigma_{\ast}(S_1 \cap S_p(K_1)) = S_p(K_2)$ and

\begin{equation} \label{eq:eq_of_degs} 
\begin{split} 
[K_2:\bQ] &= \dim_{\bQ_p} K_2 \otimes_{\bQ} \bQ_p = \sum_{\fp \in S_p(K_2)} [K_{2,\fp}:\bQ_p] = \sum_{\fp \in (S_p \cap S_1)(K_1)} [K_{1,\sigma_{\ast,K_1}^{-1}(\fp)}:\bQ_p] \\ &\leq \sum_{\fp \in S_p(K_1)} [K_{1,\sigma_{\ast,K_1}^{-1}(\fp)}:\bQ_p] = [K_1 : \bQ]. 
\end{split}
\end{equation}


By (c) we have two rational  primes $p_1, p_2$, such that $S_{p_j} \subseteq S_1$, $p_j > 2$. Let $p \in \{p_1, p_2\}$. The quotient $\Gal_{K_1, S_p}(p)$ of $\Gal_{K_1, S_1}$ is torsion-free (cf. \cite{NSW} 8.3.18 and 10.4.8). Since $K_1$ is normal over $\bQ$, $S_p$ is defined over $\bQ$ and the maximal pro-$p$-quotient of a profinite group is characteristic, we deduce that the field $K_{1, S_p}(p)$ is normal over $\bQ$. Let $L_{2,p}$ be the field corresponding to $K_{1,S_p}(p)$ via $\sigma$ (a priori, $L_{2,p}$ must not be equal $K_{2,S_p}(p)$). We have the following situation:

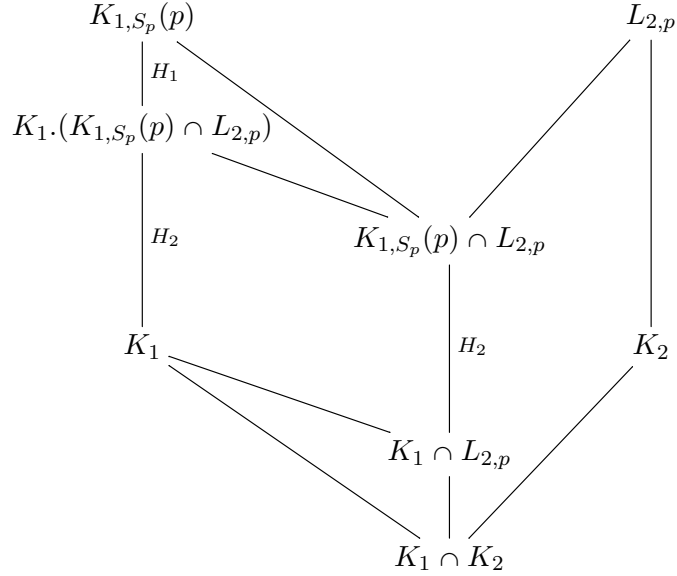
\centerline{
\begin{xy}
\xymatrix{
K_{1,S_p}(p) \ar@{-}[d]^{H_1} \ar@{-}[rdd] & & L_{2,p} \ar@{-}[ldd] \ar@{-}[ddd] \\
K_1.(K_{1,S_p}(p) \cap L_{2,p}) \ar@{-}[rd] \ar@{-}[dd]^{H_2} & &  \\
& K_{1,S_p}(p) \cap L_{2,p} \ar@{-}[dd]^{H_2}&  \\
K_1 \ar@{-}[rd] \ar@{-}[rdd] & & K_2 \ar@{-}[ldd] \\
& K_1 \cap L_{2,p} \ar@{-}[d] &  \\
& K_1 \cap K_2 &  \\
}
\end{xy}
}

\noindent In this diagram the group $H_1$ is a subgroup of $\Gal_{K_{1,S_p}(p)/K_{1,S_p}(p) \cap L_{2,p}}$ and of $\Gal_{K_1,S_p}(p)$. But the first of these two groups is finite by Proposition \ref{lm:uniform_bound} and the second is torsion-free. Hence $H_1 = 1$, i.e., $H_2 = \Gal_{K_1,S_p}(p)$. By Proposition \ref{prop:nonex_of_lifts}(a) we get $K_1 = K_1 \cap L_{2,p}$, i.e., $K_1 \subseteq L_{2,p}$. Doing this for $p = p_1, p_2$, we get: $K_1 \subseteq L_{2,p_1} \cap L_{2, p_2} = K_2$, the last equality being true, since $L_{2, p_j}/K_2$ is a pro-$p_j$-extension for $j = 1,2$. By \eqref{eq:eq_of_degs} we conclude that $K_1 = K_2$. \qed


\begin{thebibliography}{10}



\bibitem[CC]{CC} Chenevier G., Clozel L.: \emph{Corps de nombres peu ramifi\'{e}s et formes automorphes autoduales}, J. of the AMS, vol. 22, no. 2, 2009, p. 467-519.





\bibitem[Iv1]{IvDiss} Ivanov A.: \emph{Arithmetic and anabelian theorems for stable sets in number fields}, Dissertation, Universit\"at Heidelberg, 2013.

\bibitem[Iv2]{IvLocCorBound} Ivanov A.: \emph{On some anabelian properties of arithmetic curves}, preprint, \href{http://arxiv.org/abs/1309.2801}{arXiv:1309.2801}, 2013.

\bibitem[Iv3]{IvStableSets} Ivanov A.: \emph{Stable sets of primes in number fields}, preprint,  \href{http://arxiv.org/abs/1309.2800}{arXiv:1309.2800}, 2013.

\bibitem[Iv4]{IvInfDens} Ivanov A.: \emph{Densities of primes and realization of local extensions}, preprint, 2013.









\bibitem[Ne]{Ne} Neukirch J.: \emph{Kennzeichnung der p-adischen und der endlich algebraischen Zahlk\"{o}rper}, Invent. Math. \textbf{6} (1969) 296-314.
 


\bibitem[NSW]{NSW} Neukirch J., Schmidt A., Wingberg K.: \emph{Cohomology of number fields}, Springer, 2008, second edition.











\bibitem[Ta]{Ta} Tamagawa A.: \emph{The Grothendieck conjecture for affine curves}, Comp. Math. \textbf{109} (1997), 135-194.






\end{thebibliography}
\end{document}